\documentclass[10pt]{amsart}
\usepackage{amsmath,amssymb}
\usepackage[T1]{fontenc}
\usepackage{lmodern,textcomp}
\usepackage{color}
\usepackage{graphicx}
\usepackage{amsthm}
\usepackage[all]{xy}
\usepackage{verbatim}
\usepackage{enumerate}
\usepackage{caption}
\usepackage{pgf,tikz}
\usetikzlibrary{arrows}
\usepackage{geometry}
\def\H{{\mathcal{H}}}

\def\cC{{\mathcal{C}}}

\def\cG{{\mathcal{G}}}

\def\cP{\mathcal{P}}

\def\cA{\mathcal{A}}
\def\cS{\mathcal{S}}

\def\RR{\mathbb{R}}

\def\ZZ{\mathbb{Z}}
\def\NN{\mathbb{N}}
\def\QQ{\mathbb{Q}}

\def\wk{{\widetilde{\kappa}}}

\begin{document}

\theoremstyle{definition}
\newtheorem{definition}{Definition}[section]
\newtheorem{notation}[definition]{Notation}
\newtheorem{example}[definition]{Example}
\newtheorem{remark}[definition]{Remark}
\newtheorem{claim}[definition]{Claim}
\theoremstyle{plain}
\newtheorem{conjecture}[definition]{Conjecture}
\newtheorem{theorem}[definition]{Theorem}
\newtheorem{proposition}[definition]{Proposition}
\newtheorem{lemma}[definition]{Lemma}
\newtheorem{corollary}[definition]{Corollary}

\title{The Large genus asymptotic expansion of Masur-Veech volumes}
\author{Adrien Sauvaget}
\address{
Mathematical Institute, Utrecht University, Budapestlaan 6 / Hans Freudenthal Bldg, 3584 CD Utrecht, The Netherlands
}
\email{a.c.b.sauvaget@uu.nl}

\date{\today}

\maketitle

\begin{abstract}
We study the asymptotic behavior of Masur-Veech volumes as the genus goes to infinity. We show the existence of a complete asymptotic expansion of these volumes that depends only on the genus and the number of singularities.  The computation of the first term of this asymptotics expansion was a long standing problem. This problem was recently solved in~\cite{Agg} using purely combinatorial arguments, and then in~\cite{CheMoeSau} using algebro-geometric insights. Our proof relies on a 
combination of both methods. 
\end{abstract}

\section{Introduction}

\subsection{Masur-Veech volumes} Let $\mu=(k_1,\ldots, k_n)$ be a vector of positive integers. In the text we denote by  $|\mu|$ and $m(\mu)$ the sum and products of the entries of $\mu$ and by $n(\mu)=n$ its length.  The {\em space  of holomorphic differentials of type $\mu$}   is the moduli space $\H(\mu)$ of objects $(C,x_1,\ldots,x_n, \omega)$, where:
\begin{itemize}
\item  $C$ is a connected Riemann surface with $n$ distinct marked points $x_i$;
\item  $\omega$ is a holomorphic 1-form on $C$ such that ${\rm ord}_{x_i}\omega= k_i-1$ for all $1\leq i\leq n$.
\end{itemize}
This moduli space has an integral affine structure (defined in period coordinates) and thus  a natural measure $\nu$. The {\em Masur-Veech volume} is defined as
$$
{\rm Vol}(\mu)= 2 |\mu| \cdot \nu \left(\H_{\leq 1}(\mu)\right)
$$
where $\H_{\leq 1}(\mu)\subset \H(\mu)$ is the subspace of points $(C,x_1,\ldots,x_n,\alpha)$ such that 
$$
\frac{i}{2} \int_{C} \omega\wedge \overline{\omega} \leq 1.
$$
The real number ${\rm Vol}(\mu)$ is finite (see~\cite{Mas} and~\cite{Vee}). For the purpose of the paper we define the {\em reduced volumes} to be 
$$
v(\mu)=m(\mu) \cdot {\rm Vol}(\mu).
$$
The {\em genus} of $\mu$ is defined as  $g(\mu)=(|\mu|-n(\mu)+2)/2$. We say that $\mu$ is an {\em admissible profile} if $g(\mu)$ is an integer.
The genus of a Riemann surface in $\H(\mu)$ is $g(\mu)$, thus $v(\mu)\neq 0$ if and only if $\mu$ is an admissible profile.

\subsection{Large genus analysis}
The purpose of the paper is to study the behaviour of the reduced volumes as the genus goes to infinity. We will use the following convention.
\begin{itemize}
\item We denote by ${\rm Adm}$ the set of admissible profiles. For all $r\in \ZZ_{\geq 0}$, we denote by ${\rm Adm}(r)\subset {\rm Adm}$ the set of profiles with entries greater 
or equal 
than $r$.
\item Let $f:\ZZ_{>0}\to \RR$ be a function and $r \in \ZZ_{>0}$. The notation  $O_r(f(g))$ stands for: a function $\epsilon:{\rm Adm}(r)\to \RR$, such that there exists a constant $C>0$ satisfying for all $\mu\in {\rm Adm}(r)$, $\epsilon(\mu)< C\cdot f(g(\mu))$.
\item For all $(x,\ell)\in \NN^2$, we denote  $(x)_\ell=x(x-1)\ldots (x-\ell+1)$ (the {\em decreasing Pochhammer symbol}).
\end{itemize}

\begin{theorem}\label{th:main}
There exists a family of  coefficients $(c_{k,\ell})_{(k,\ell)\in \NN^2}$ in $\QQ[\pi^2]$ such that for all $r\in \NN$, we have
\begin{equation}\label{for:main}
v(\mu) = \sum_{k+\ell\leq r} \frac{c_{k,\ell}}{g(\mu)^k (|\mu|-1)_\ell} + O_r(g^{-r}).
\end{equation}
Moreover, these coefficients are uniquely determined and can be explicitly computed.
\end{theorem}
The first coefficients of the asymptotic expansion are given by
$$
v(\mu)=4+ \left[\frac{-2\pi^2}{3\, (|\mu|-1)}\right] +\left[ \frac{-27\pi^2+\pi^4}{72\, g^2}+ \frac{\pi^{2}}{6\, (|\mu|-1)_2}\right]+\ldots 
$$
The first two orders have been computed in~\cite{Agg} and~\cite{CheMoeSau}.

\subsection{Comments and related works} Recently, the large genus asymptotics of numerical  invariants associated to moduli spaces of curves has interested geometers and physicists. For example, the asymptotic expansion of Weil-Peterson volumes was studied in~\cite{MirZog}, and~\cite{LiuMulXu}, while the asymptotic bahavior of Gromov-Witten invariants and Hurwitz numbers was studied in~\cite{Zin}, or~\cite{Dub2}. 

These problems often admits a counterpart in terms of asymptotic dynamical properties of random surfaces (see~\cite{MirPet} or~\cite{Mir2}), or in string theory: one can either consider systems with a large number of particles (see~\cite{GhoRaju}) or the asymptotic behavior of perturbative expansions (see~\cite{CouMarSch}). 
\bigskip

Theorem~\ref{th:main} can be seen as the analogue of the results of Zograf and Mirzakhani for Weil-Peterson volumes (of moduli of hyperbolic structures) in the context of flat geometry. However two special features for Masur-Veech volumes can be noticed.
\begin{itemize}
\item The notion of convergence that we use in the text may seem restrictive at first sight as we only consider flat surfaces with sufficiently big singularities. However, this restriction is already needed in the case $r=2$  (see~\cite{CheMoeSau}). Indeed, we have $v(\mu+(1))=v(\mu)$, thus the first correction term ($-\frac{2\pi^2}{3(|\mu|-1)}$) can only be computed when we impose that all entries of $\mu$ are greater than $1$. More generally, the behavior of Masur-Veech volumes in the small singularities range is piece-wise linear. Therefore the asymptotic expansion of Masur-Veech volumes of general strata cannot be expressed with symmetric rational functions.
\item As for the Weil-Peterson volumes, the asymptotic expansion of reduced Masur-Veech volumes is computed in terms of $g$ and $n$  (in the large singularity range). However, $v$ is a function defined on the set of partition. Thus, it is somehow surprising that  its asymptotic behavior only depends on the first two symmetric functions (length and size). 
\item Here we use two equations to compute the coefficient of the asymptotic expansion: the induction formula for minimal strata of~\cite{Sau4} and the volume recursion of~\cite{CheMoeSau}. The graphs underlying this induction formulas are of compact type and the terms involved are very rapidly divergent series. This makes the computation of the asymptotic expansion of Masur-Veech volume easier than the one of Weil-Peterson volume which is based on the Virasoro constraint and the topological recursion (see~\cite{MirZog}. 
\end{itemize}

Finally, let us mention that  A. Eskin and A. Zorich proposed a series of four conjectures about the limits of numerical invariants of strata of abelian differentials: they considered Masur-Veech volumes, area Siegel-Veech constants, and refinements of these two functions according to spin parity (see~\cite{EskZor}). All conjectures were solved in~\cite{Agg},~\cite{Agg1} and~\cite{CheMoeSau}. Using the arguments of the present text, one can show that these four functions admit asymptotic expansions in the spirit of~Theorem~\ref{th:main}.  However, we only consider the Masur-Veech volumes to keep the presentation short and clear.

\subsection*{Acknowledgements} I would like to thank A. Aggarwal, D.~Chen, M.~Mo\"eller, and A.~Zorich for useful discussions on the subject. 

\section{The volume recursion}

It has been shown in~\cite{CheMoeSau} and~\cite{Sau4} that the Masur-Veech volumes can be computed inductively. In this section we use these induction formulas to state two estimates on Masur-Veech volumes.

\subsection{Notation}  Before presenting the volume recursion, we fix some general notation for partitions and compositions.

 Let $n$ be a non-negative integer and $m$ a positive integer. A   {\em composition}  (resp. {\em non-negative composition}) 
 is a $m$-uple of positive (resp. non-negative) 
 integers $\underline{d}=(d_1,\ldots,d_m)$ summing to $n$. We denote by $\cC_n(m)$ and $\cG_n(m)$ 
 the sets of compositions and non-negative compositions 
 of $n$ of length $m$. 

Let $E$ be a set and $m\geq 1$. A {\em partition} of $E$ of length $m$ is a list of $m$ disjoint sub-sets $\alpha=(\alpha_1,\alpha_2,\ldots,\alpha_m)$ of $E$ such that $E=\alpha_1\sqcup \alpha_2\ldots \sqcup \alpha_m$. We denote by $\cP(E)_m$, the set of partitions of $E$ of length $m$. If $E$ is of cardinal $n$ then $\cP(E)_m$ has cardinal $m^n$. If $\underline{d}$ is a non-negative partition of $n$, then we denote by $\cP(E,\underline{j})$ the set of partitions $\alpha$ such that the cardinal of $\alpha_i$ is $d_i$ for all $1\leq i\leq m$. 

If $\mu=(k_1,\ldots,k_n)$ be a vector of positive integers and $E=\{i_1,\ldots,i_\ell\}$ is a subset of $[\![1,n]\!]$ then we denote by $\mu_{E}$ the vector $(k_{i_1},\ldots,k_{i_\ell})$. Moreover if $\mu$ and $\mu'$ are two vectors of positive integers, then we denote by $\mu+\mu'$ the concatenation of the two vectors.

We finish this section by recalling the following inequality between factorial numbers proved in~\cite{Agg} (see Lemma~2.5).
\begin{lemma}\label{lem:factorial}
Let $A_1, A_2,C_1$ and $C_2$ be non-negative integers. We have the following inequality
$$
(A_1+C_1)!(A_2+C_2)!\leq \left(A_1+A_2+{\rm max}(C_1,C_2)\right)! \ \left({\rm min}(C_1,C_2)\right)!
$$
\end{lemma}

\subsection{Induction formulas} In order to state the induction formulas for Masur-Veech volumes, we introduce the following notation:
$$
a(\mu) = \frac{|\mu|! }{2(2\pi)^{2g(\mu)}}\cdot v(\mu) \in \QQ.
$$
Besides, we define the two following formal series in $\QQ[[t]]$:
$$
\cA(t)=1+\sum_{g\geq 1} a(2g-1) t^{2g}, \text{ and }\, \cS(t)=\frac{t/2}{{\rm sinh}(t/2)}.
$$
\begin{proposition}
The numbers $a(\mu)$ are determined by the two following induction relations.
\begin{itemize}
\item  For all $g\geq 1,$ we have:
\begin{equation}\label{for:ind1}
[t^{2g}] \cA(t)^{2g}= (2g)! [t^{2g}] \cS(t);
\end{equation}
\item If $\mu=(k_1,\ldots,k_n)$ has length at least two, then
\begin{equation}\label{for:ind2}
a(\mu)=\sum_{m>0}\frac{|\mu|}{m}\!\!\!\!\!\!\!\!\!\! \sum_{\begin{smallmatrix}
\alpha\in \cP([\![3,n]\!])_m,  \\
\underline{d}\in\cC_{k_1}(m), \underline{d'}\in\cC_{k_2}(m) \end{smallmatrix}} \!\!\!\!\!\!\!\!\!\! \Bigg( \prod_{i=1}^m  a(\mu_{\alpha_i}+(d_i+d_i'-1))\Bigg).
\end{equation}
\end{itemize}
\end{proposition}
Using formula~\eqref{for:ind1}, we already proved in~\cite{Sau4} the following asymptotic expansion of $v(2g-1)$.
\begin{proposition}\label{prop:kappa}
There exists a family of coefficients $(\kappa_i)_{i\in \NN}$ in $\QQ[\pi^2]$ such that for all $r\in \NN$, we have
\begin{equation}\label{for:estkappa}
v(2g-1)=\sum_{0\leq i\leq r} \frac{\kappa_i}{g^i} + O(g^{-r-1}).
\end{equation}
\end{proposition}

Now let $\mu=(k_1,\ldots,k_n)$ be a profile of length at least two. The value of the total sum in the right-hand side of Formula~\eqref{for:ind2} does not depend on the order of the $k_i's$, however each summand does depends on the choice of the two first terms. Thus, in all the text we will assume that $k_1$ and $k_2$ are the smallest entries of $\mu$ (this assumption is not necessary but simplifies the proof of certain estimates). 

For $k\geq 1$ we denote by $\mu^{(k)}$ the profile $(k_1+k_2-1-2k,k_3,\ldots,k_n)$. The main purpose of the section will be to prove the following estimates (using Formula~\eqref{for:ind2}). 
\begin{proposition}\label{prop:wkappa}
There exists a family of coefficients $(\wk_i)_{i\in \NN^*}$ in $\QQ[\pi^{2}]$ such that for all $r\in \NN$ and all $\mu\in {\rm Adm}(r)$ of length at least two, we have
\begin{equation}\label{for:estwkappa}
{v(\mu)}=v(\mu^{(0)})+ \wk_1\cdot \frac{v\left(\mu^{(1)}\right)}{(2g-3+n)_2}+\wk_2\cdot \frac{v\left(\mu^{(2)}\right)}{(2g-3+n)_4} +\ldots + O_r(g^{-r-1}). 
\end{equation}
\end{proposition}

\subsection{Splitting the right-hand side of~\eqref{for:ind2}}

In order to prove Proposition~\ref{prop:wkappa}, we begin by splitting the sum in the right-hand side of formula~\eqref{for:ind2} in several parts that will be estimated separately.  
We define the following functions of $\mu$.
\begin{itemize}
\item For all $1\leq m\leq {\rm min}(k_1,k_2)$, we denote:
$$
A_m(\mu)=\sum_{\begin{smallmatrix}
\alpha\in \cP([\![3,n]\!])_m,  \\
\underline{d}\in\cC_{k_1}(m), \underline{d'}\in\cC_{k_2}(m) \end{smallmatrix}} \!\!\!\!\!\!\!\!\!\! \Bigg( \prod_{i=1}^m  a(\mu_{\alpha_i}+(d_i+d_i'-1))\Bigg).
$$
With this notation we have $a(\mu)=\sum_{m\geq 1} \frac{A_m(\mu)}{m}$.
\item For all $\mu$, $m\geq 2,$ and $0\leq \ell\leq n-2$, we denote 
$$
A^\ell_m(\mu)= \!\!\!\!\!\!\!\! \sum_{\begin{smallmatrix}
\alpha\in \cP([\![3,n]\!],(\ell,(n-2-\ell)),  \\
1\leq d\leq k_1, 1\leq d'\leq k_2 \end{smallmatrix}} \!\!\!\!\!\!\!\!\!\! a((d+d'-1)+\mu_{\alpha_1}) \cdot A_{m-1}((k_1-d,k_2-d')+(\mu_{\alpha_2}),
$$
and we have $A_{m}(\mu)=\sum_{0\leq \ell\leq n-2} A_{m}^{\ell}(\mu)$.
\item Finally for all $\mu, m, \ell,$ and $2\leq D\leq k_1+k_2-2m$, we denote
$$
A^{\ell,D}_{m}(\mu)= \!\!\!\!\!\!\!\! \sum_{\begin{smallmatrix}
\alpha\in \cP([\![3,n]\!],(\ell,(n-2-\ell)),  \\
1\leq d\leq k_1, 1\leq d'\leq k_2\\
d+d'=k_1+k_2 -D \end{smallmatrix}} \!\!\!\!\!\!\!\!\!\! a((d+d'-1)+\mu_{\alpha_1}) \cdot A_{m-1}((k_1-d,k_2-d')+(\mu_{\alpha_2})),
$$ 
and we have $A^{\ell}_m(\mu)=\sum_{D\geq 2} A^{\ell,D}_m(\mu)$.
 \end{itemize}
 Using this extra set indices, we write $\wk_i'= (2\pi)^{2i} \sum_{m\geq 2}\sum_{d+d'=2i} A_{m-1}(d,d')$ for all $i\geq 1$. We shall prove that these coefficients satisfy the assumption of Proposition~\ref{prop:kappa}. The following three Lemmas provide estimates of partial sums of the $A^{\ell,D}_m(\mu)$.
\begin{lemma}\label{lem:est1}
There exists a positive constant $C$ such that for all $\mu$ and for all $m\geq 2$, we have
$$
\frac{A_m(\mu)}{a(\mu)}< \frac{m C^{m}}{|\mu|^{2m-2}}.
$$ 
\end{lemma}

\begin{proof}
The Proof is given in~\cite{Agg1} (see Proposition~3.4). 
\end{proof}

\begin{lemma}\label{lem:est2}
Let $D>0$. There exists a constant $C>0$, such that for all  $\mu$ and $m\geq 2$, we have
$$
 \frac{{A}^{n-2,D}_m(\mu) + {A}^{0,k_1+k_2-D}_m(\mu)}{a(\mu)} < \frac{ C}{|\mu|^{D}}
$$
\end{lemma}

\begin{proof}
Let us fix $D>0$. The set of triples $(m,d,d')$ such that $d+d'\leq D$ and $2\leq m\leq \rm{min}(d,d')$ is finite. Therefore, we can set $M={\rm max}_{m,d+d'\leq D} A_{m-1}(d,d')$. For all $\mu$ and $m$, we have
$$
A_{m}^{n-2,D}(\mu)\leq \sum_{\begin{smallmatrix}
1\leq d\leq k_1, 1\leq d'\leq k_2\\
d+d'=k_1+k_2 -D \end{smallmatrix}} M\cdot a(\mu^{(D/2)})\leq D\cdot M\cdot a(\mu^{(D/2)})
$$
(if $D$ is odd then the number $A_{m}^{\ell,D}(\mu)$ vanishes). By similar arguments one can show that there exists a constant $M'$ such that ${A}^{0,k_1+k_2-D}_m(\mu) \leq  D\cdot M' a(\mu^{(D/2)})$, thus finishing the proof.
\end{proof}

\begin{corollary}\label{cor:two}
For all $r\geq 2$. There exists a constant $C$ such that for all $\mu$ of length $2$, we have
$$
\left| v(\mu)- v(\mu^{(0)})-\!\!\!\! \sum_{1\leq i<(r-3)/2} \wk_i' v(\mu^{(i)})\right| < C \cdot g^{-r-1}.
$$
\end{corollary}

\begin{proof}
We introduce the following notation
$$
A_{m}'(\mu)=\sum_{\begin{smallmatrix}\underline{d}\in \cC_{k_1}(m), \underline{d'}\in \cC_{k_2}(m)\\
{\rm max}_{1\leq i\leq m} d_i+d_i'< k_1+k_2-r\end{smallmatrix}}
\prod_{i=1}^n a(d_i+d_i'-1).
$$
If $g$ is large enough (bigger than $m\times r$) then 
$$
A_{m}'(\mu)\leq m \sum_{r<D<k_1+k_2-r} A_{m}^{0,D}(\mu).
$$
Thus there exists a constant $K$ such that for all $\mu$ of length $2$ and $m<(r-3)/2$, we  have
$$
\frac{A_{m}'(\mu)}{a(\mu)}<\frac{K}{g^{r+1}}
$$
Now for $g$ large enough we have the identity: 
$$
A_m(\mu)=A_m'(\mu)+ m \sum_{D\leq r} A_{m}^{0,D}(\mu).
$$ Indeed to sum over all compositions of $k_1+k_2$ of length $m$  with a big maximal value, one only need to chose the value and the position $1\leq i\leq m$ of this maximum. Thus summing over all values of $m$, there exists a constant $C>0$ such that:
\begin{equation}\label{for:kappap}
\left| a(\mu)-a(\mu^{(0)}) - \sum_{2\leq m<(r-3)/2}\sum_{D\leq r} A_{m}^{0,D}(\mu)  \right| <\frac{C}{g^{r+1}} a(\mu).
\end{equation}
The sum in the left-hand side can be rewritten as
$$
\sum_{D\leq r} a(\mu^{(D/2)}) \left(\sum_{m\geq 2}\sum_{d+d'=D} A_{m-1}(d,d')\right)= \sum_{D\leq r} a(\mu^{(D/2)}) \frac{\wk_{D/2}'}{(2\pi)^{D}}
$$
(recall that for odd $D$, we have $A_m^{0,D}=0$). We get the desired inequality by multiplying~\eqref{for:kappap}  by $2(2\pi)^{2g}/(2g-2+n)!$.

\end{proof}

\begin{lemma}\label{lem:est3}
We fix $r\geq 2$. We denote $\widehat{A}_m(\mu)=\sum_{1\leq \ell \leq n-3} A_m^\ell(\mu)$. There exists a constant $C>0$, such that for all  $\mu\in {\rm Adm}(r)$ and $2\leq m< (r+3)/2 $, we have
$$
\frac{\widehat{A}_m(\mu)}{a(\mu)}< \frac{ C}{|\mu|^{r+1}}.
$$
\end{lemma}

\begin{proof}
We begin by remarking that $\widehat{A}_m(\mu) \leq A_{m}(\mu)$. Therefore Lemma~\ref{lem:est1} implies that there exists a constant $K'>0$ such that: for all $1\leq m< (r+3)/2 $, we have $A_{m}(\mu)< K' A_1(\mu)$ and 
$$
A^\ell_{m}(\mu)< K'\cdot A^{\ell}_{2}(\mu)= K' \cdot \!\!\!\!\!\! \!\!\!\!\!\! \!\!\!\!\!\!\!\! \sum_{\begin{smallmatrix}
\alpha\in \cP([\![3,n]\!],(\ell,(n-2-\ell)),  \\
1\leq d\leq k_1, 1\leq d'\leq k_2 \end{smallmatrix}} \!\!\!\!\!\! \!\!\!\!\!\!\!\!\!\! a((d+d'-1)+\mu_{\alpha_1}) \cdot a((k_1+k_2-d-d'-1)+\mu_{\alpha_2}).
$$
Thus in order to estimate the $\widehat{A}_m(\mu)$ we only need to estimate $\widehat{A}_2(\mu)$. We use the expression of $a(\mu)$ and the boundedness of the $v(\mu)$ to deduce that there exists a constant $K$ such that 
\begin{equation}\label{for:ineqA2}
(2\pi)^{2g} A_{2}^{\ell}(\mu) \leq K \cdot  \!\!\!\!\!\!\!\!\!\!\!\!\!\!\!\! \sum_{\begin{smallmatrix}
\alpha\in \cP([\![3,n]\!],(\ell,(n-2-\ell)),  \\
1\leq d\leq k_1, 1\leq d'\leq k_2 \end{smallmatrix}}  \!\!\!\!\!\!\!\! \!\!\!\!\!\!\!\!\!\! \left(d+d'-1+|\mu_{\alpha_1}|\right)! \left((k_1+k_2-d-d'-1)+|\mu_{\alpha_2}|\right)!.
\end{equation}
Now for all $\ell$ we will estimate $A^\ell_{2}$ by bounding the contribution of all terms in the sum of the right-hand side. To do so, we will the cases  $\ell=1$ or $n-3$ from the others. We set:
$$\widetilde{A}_2(\mu)=\sum_{2\leq \ell \leq n-4} A_2^\ell(\mu)=\widehat{A}_2(\mu)-A_2^1(\mu)-A_2^{n-3}(\mu).
$$

\noindent \underline{\em Estimating $A_{2}^{\ell}(\mu)$ for $2\leq \ell\leq n-4$.} For all $\alpha,d,d'$ in the range of the sum of~\eqref{for:ineqA2}, we apply Lemma~\ref{lem:factorial} with 
\begin{eqnarray*}
&& A_1=|\mu_{\alpha_1}|+d+d'-1- r\ell, \, A_2=|\mu_{\alpha_2}|+k_1+k_2-(d+d')-1- r(n-2-\ell), \\
&&   C_1=r\ell,  \text{ and }\, C_2=r(n-2-\ell).
\end{eqnarray*}
Then the terms of the sum~\eqref{for:ineqA2} are bounded by 
$$
\left\{ \begin{array}{cl} (|\mu|-2- r\ell)! (r\ell)! & \text{if $\ell\leq (n-2)/2$,}\\ (|\mu|-2- r(n-2-\ell) )! (r(n-2-\ell))! & \text{otherwise} \end{array}\right.
$$

Then summing over all $2\leq \ell\leq n-4$, we get 
\begin{eqnarray*}
(2\pi)^{2g} \widetilde{A}_m(\mu)&\leq & 2K \cdot  k_1k_2\cdot \!\!\!\!\! \sum_{2\leq \ell\leq (n-2)/2} \binom{n-2}{\ell} (|\mu|-2-r\ell)! (r\ell)!\\
&=& 2K \cdot k_1 k_2(n-2)! \cdot \!\!\!\!\! \sum_{2\leq \ell\leq (n-2)/2} \frac{(r\ell)!}{\ell!(n-2-\ell)!} (|\mu|-2-r\ell)! 
\end{eqnarray*}
The term in the sum is a decreasing function of $\ell$: indeed, for $\ell\geq 3$ we have:
\begin{eqnarray*}
&&\left(\frac{(r\ell)!}{\ell!(n-2-\ell)!} (|\mu|-2-r\ell)!)\right) \cdot \left(\frac{(r\ell-r)!}{(\ell-1)!(n-1-\ell)!} (|\mu|-2-r\ell+r)! \right)^{-1}\\
&=& \frac{r\cdot (n-1-\ell)}{(|\mu|-1-r\ell+r)}\cdot \frac{(r\ell-1)\times \ldots \times (r\ell-r+1)}{(|\mu|-2-r\ell+r) \times \ldots \times (|\mu|-r\ell-1)}\leq 1\cdot 1.
\end{eqnarray*}
The two inequalities here come from: $r\cdot n\leq |\mu|$ (as all entries of $\mu$ are greater than $r-1$) and $r\ell-1\leq |\mu|-r\ell-1$.
Thus all the terms of this sum are bounded by
$$
(|\mu|-2-2r )! \frac{(2r)!}{r!}.
$$
Finally we use the fact that $k_1$ and $k_2$ are the smallest entries of $\mu$ to deduce that $k_1k_2(n-2)(n-3)\leq (|\mu|-2r)(|\mu|+1-2r)$ and that $(n-3)\leq |\mu|+2-2r$ to deduce that
$$
(2\pi)^{2g} \widetilde{A}_m(\mu)\leq K \frac{(2r)!}{r!}  (|\mu|+2-2r)!\leq K \frac{(2r)!}{r!}  (|\mu|-1-r)!
$$
as $r\geq 2$. We deduce from this estimates that there exists a constant $C'$ such that  
$$\frac{\widetilde{A}_m(\mu)}{a(\mu)}\leq \frac{C'}{ |\mu|^{r-1}}.$$

\noindent \underline{\em Estimating $A_{2}^{1}(\mu)$ and $A_{2}^{n-3}(\mu)$.} In order to bound both cases we make a new disjonction. 

First we assume that $k_1=k_2=r$. Then the above estimates for $2\leq \ell\leq n-4$ give:
\begin{eqnarray*}
(2\pi)^{2g}\left(A_{m}^{1}(\mu)+A_{2}^{n-3}(\mu)\right)&\leq& 2 K\cdot k_1k_2 \binom{n-2}{1} (|\mu|-2-r)! \\
&\leq & (2K r^2) (|\mu-1-r)! 
\end{eqnarray*}

Otherwise we use the assumption that $k_1$ and $k_2$ are the smallest entries. Thus all other entries are at greater than $r$. Therefore for all non-trivial partitions $\alpha$ of $[\![3,n]\!]$ into two non-trivial sets, then both $|\mu_{\alpha_1}|$ or $|\mu_{\alpha_2}|$ are larger than $r$. Thus each summand in the right-hand side of the inequality~\eqref{for:ineqA2} is smaller than:
$$
(r+1)! (|\mu|-r-3)!
$$
Then we have the inequality
\begin{eqnarray*}
(2\pi)^{2g}\left(A_{m}^{1}(\mu)+A_{2}^{n-3}(\mu)\right)&\leq& 2 K\cdot k_1k_2 \binom{n-2}{1} (r+1)! (|\mu|-r-3)!\\
&\leq& 2K \cdot (|\mu|-r-1)!
\end{eqnarray*}
Thus there exists a constant $C''$ such that for all $\mu$ and $m\geq 2$, we have
$$
A_{m}^{1}(\mu)+A_{2}^{n-3}(\mu)\leq C''\frac{a(\mu)}{|\mu|^{r+1}}.
$$
All together the constant $C=C'+C''$ satisfies the requirement of the lemma. 
\end{proof}

If $E$ is a set, we denote by $\overline{\cP}(E)\subset {\cP}(E)$ the set of partitions with at least two non-empty parts. For all $m\geq 1$, we denote
$$
\overline{A}_m(\mu)=\sum_{\begin{smallmatrix}
\alpha\in \overline{\cP}([\![3,n]\!])_m,  \\
\underline{d}\in\cC_{k_1}(m), \underline{d'}\in\cC_{k_2}(m) \end{smallmatrix}} \!\!\!\!\!\!\!\!\!\! \Bigg( \prod_{i=1}^m  a(\mu_{\alpha_i}+(d_i+d_i'-1))\Bigg).
$$
\begin{lemma}\label{lem:est4}
We fix $r\geq 2$.  There exists a constant $C>0$, such that for all  $\mu\in {\rm Adm}(r)$ and $2\leq m< (r+3)/2 $, we have
$
\overline{A}_m(\mu)<C \frac{ {a(\mu)}}{|\mu|^{r+1}}.$
\end{lemma}
\begin{proof} We prove this Lemma by induction on $r\geq 1$. The case $r=1$ is clear as $\overline{A}_m(\mu)=0$ for $m=1$. Now, let $r\geq 2$. By the induction assumption, there exists a constant $K$ such that for all $r'<r$, all $\mu\in {\rm Adm}(r')$ and $m<(r+3)/2$, we have $A_m(\mu)< K a(\mu)/|\mu|^{r+1}$ (such a $K$ exists because there are finitely many $0<r'<r$ to which we apply the induction hypothesis). We begin by writing
\begin{eqnarray*}
\overline{A}_m(\mu)&=&\widehat{A}_m(\mu)+\sum_{\begin{smallmatrix} d\leq k_1,d'\leq k_2 \end{smallmatrix}} a(d+d'-1)\cdot \overline{A}_{m-1}\left((k_1-d,k_2-d')+\mu_{[\![n-3]\!])}\right).\\
&=& \widehat{A}_m(\mu)+\sum_{\begin{smallmatrix} d\leq k_1,d'\leq k_2\\ d+d'> r \end{smallmatrix}} a(d+d'-1)\cdot \overline{A}_{m-1}\left((k_1-d,k_2-d')+\mu_{[\![n-3]\!])}\right).\\
&+&\sum_{\begin{smallmatrix} d\leq k_1,d'\leq k_2\\ d+d'\leq r \end{smallmatrix}} a(d+d'-1)\cdot \overline{A}_{m-1}\left((k_1-d,k_2-d')+\mu_{[\![n-3]\!])}\right).
\end{eqnarray*}
Lemmas~\ref{lem:est2} and~\ref{lem:est3} imply that there exists a constant $C'$ such that the first two terms are smaller than $C' a(\mu)/g^{r+1}$ for $\mu$ and $m$ thus we will only consider the third one that we denote by $\overline{A}_m'(\mu)$. We denote by $M={\rm max}_{r'\leq r-2} a(r')$. Then we have
$$
\overline{A}_m'(\mu)\leq M\cdot\!\!\!\!\!\! \sum_{\begin{smallmatrix} d\leq k_1,d'\leq k_2\\ d+d'\leq r \end{smallmatrix}} \!\!\!\!\!\!  \overline{A}_{m-1}\left((k_1-d,k_2-d')+\mu_{[\![n-3]\!])}\right).
$$
In this sum, if $d+d'=D$, then $k_1-d$ and $k_2-d'$ are greater or equal to $r-D$ and $(m-1)< (r-3)/2$. Thus we can write: 
$$
\overline{A}_{m-1}\left((k_1-d,k_2-d')+\mu_{[\![n-3]\!])}\right) \leq K a\left((k_1-d,k_2-d')+\mu_{[\![n-3]\!])}\right)
$$
\end{proof}

\begin{proof}[End of the proof of Proposition~\ref{prop:wkappa}]
We fix $r>0$. We consider only vectors $\mu\in {\rm Adm}(r)$ of length greater than $2$. Using Lemma~\ref{lem:est4}, there exists a constant $C$ such that 
$$\frac{A_m(\mu)-m A_{m}^{n-2}(\mu)}{a(\mu)}< C g^{-r-1}.$$
Indeed
 $\frac{\overline{A}_{m-1}(\mu)}{a(\mu)}=O(g^{-r-1})$ and 
 ${A}_{m-1}(\mu)- \overline{A}_{m-1}(\mu)=mA_{m}^{n-2}(\mu){a(\mu)}$
  (as summing over all partitions of $[\![3,n]\!]$ of length $m$ with exactly one non-trivial part is the same as choosing the index in $[\![1,m]\!]$ that has the only non-trivial part). Now use Lemma~\ref{lem:est2} to write
$$
A^{n-2}_m(\mu)=\sum_{D\leq r} A^{n-2,D}_m(\mu) + O_r(g^{-r-1})
$$
Thus there exists a constant $C$ such that 
$$
\left| a(\mu)- |\mu| \sum_{m\geq 1}\sum_{D\leq r} A_{m}^{n-2,D}(\mu)\right|< \frac{C}{g^{r+1}}
$$
Now we use the fact that $A_{m}^{n-2,D}(\mu)=a(\mu^{(D/2)}) \cdot \sum_{d+d'=D} A_{m-1}(d,d')$, to obtain the inequality:
\begin{eqnarray*}
\bigg| (2\pi)^{2g}&\!\!\!\!\! \!\!\!\!\! &(a(\mu)-a(\mu^{(0)}) \\
&\!\!\!\!\! \!\!\!\!\! &- |\mu| \!\!\! \sum_{2\leq D\leq r} (2\pi)^{2g-D}a(\mu^{(D/2)}) \sum_{m>1} \sum_{d+d'=D} (2\pi)^{2D} A_{m-1}(d,d')\bigg|< \frac{C (2\pi)^{2g}}{g^{r+1}}
\end{eqnarray*}
Now we divide this inequality by $|\mu|!/2$ to obtain:
$$
\left|v(\mu)-v(\mu^{(0)})- \sum_{i\geq 2} \wk_i'v(\mu^{(i)})\right|< \frac{C}{g^{r+1}}.
$$
Together with Corollary~\ref{cor:two} for the case $n=2$, we get Proposition~\ref{prop:wkappa}. 
\end{proof}

\section{Computing the asymptotic expansion by induction}

In this section use the estimates~\eqref{for:estkappa} and~\eqref{for:estwkappa} to complete the proof of Theorem~\ref{th:main}. To do so, we recall  some basic tools of linear algebra for linear combinations of Pochammer.

\subsection{Linear combinations of Pochammer symbols} Here we consider the algebra rational functions of two variables $g$ and $n$.  We define the vector space $H\subset \QQ(g,n)$ as the sub-vector space spanned by elements of the form
$$
Q_{m,\ell}(g,n)= \frac{1}{g^m(2g-3+n)_\ell},
$$
for $(m,\ell)$ in $\NN^{2}$. Obviously the $Q_{m,\ell}$ are linearly independant in $\QQ(g,n)$. The vector space $H$ is naturally bi-graded by $(m,\ell)\in \NN^2$: we denote  by $H^{m,\ell}\subset H$ the vector space spanned by $Q_{m,\ell}$. Finally we denote respectively by $H^{r}, H^{\leq r},$ and $H^{\geq r}$ the spaces $H^{r}=\bigoplus_{m+\ell=r} H^{m,\ell}$ (or $\leq r$, $\geq r$). 
\bigskip

For $Q\in H$, we define the rational function $\Delta(Q)(g,n)=Q(g,n)-Q(g,n-1)$. For all $(m,\ell)\in \NN^{2}$, we have $\Delta(Q_{m,\ell})=-\ell Q_{m,\ell+1}$. Therefore, the function $\Delta$ is a linear endomorphism of $H$ of bi-degree $(0,1)$ and its kernel is $\bigoplus_{m} H^{m,0}$. We define the inverse of $\Delta$ as
\begin{eqnarray*}
\Delta^{-1}: \bigoplus_{m,\ell\geq 2} H^{m,\ell} &\to& \bigoplus_{m,\ell\geq 1} H^{m,\ell}\\
Q_{m,\ell} &\mapsto &  -\frac{Q_{m,\ell-1}}{\ell-1}.
\end{eqnarray*} 

%

\begin{lemma}\label{lem:unipoch}
Let $r\geq 0$ and $Q\in \RR\otimes H^{r}$.  If there exist constants $g_0,C,\lambda>0$ such that $Q(g,n)< \frac{C}{g^{r+1}}$ for all integers $g\geq g_0$, and $1\leq n\leq \lambda g$, then $Q=0$.
\end{lemma}
\begin{proof} 
We prove the lemma by induction on $r$. The base case $r=0$ is obvious, thus we assume that $r>0$ and that the lemma holds for all $r'\leq r-1$.

Let $Q=\sum_{m+\ell\leq r}  c_{m,\ell} Q_{m,\ell}$ be a rational function satisfying the hypothesis of the Lemma. In particular the function $\widetilde{Q}=\sum_{m+\ell\leq r-1}  c_{m,\ell} Q_{m,\ell}$ satisfies the hypothesis of the lemma for $r'=r-1$. Thus the coefficients $c_{m,\ell}$ vanish for $m+\ell<r$. Therefore we can assume that $Q\in \RR\otimes H^r$.

We write $Q=\sum_{m+\ell\leq r}  c_{m,\ell} Q_{m,\ell}$. We assume that $Q\neq 0$. Let $m$ be the smallest integer $0\leq m\leq r$ such that $c_{m,r-m}\neq 0$. We consider the rational function $F(g,n)=Q/(c_{m,r-m} Q_{m,r-m})$. Let $s$ be a positive rational number smaller than $s<\lambda$. There exists an infinite number of $g$ such that $sg$ is an integer. Besides for $m'\geq m$, we have
$$\underset{g\mapsto \infty}{\lim} Q_{m',r-m'}(g,sg)/Q_{m,r-m}(g,sg) ={(2+s)^{m-m'}}.$$ Therefore the limit of $F(g,sg)$ as $s$ goes to infinity is a non-zero polynomial in $(2+s)$. However, $Q_{m,r-m}(g,sg)<C g^{-r}$ for all $g$ and $s<\lambda$ thus  the limit of $F(g,sg)$ is null. 
\end{proof}

\subsection{Computation of the coefficients of the asymptotic expansion~\eqref{for:main}}

In this section we will prove the following proposition.
\begin{proposition}\label{prop:main} For all $r\geq 1$, there exists a rational function $Q_r\in \QQ[\pi^{2}]\otimes H^{\leq r}$ such that $v(\mu)=Q_r(g,n)+O_r(g^{-r})$.
\end{proposition}
This proposition implies Theorem~\ref{th:main}. Indeed, Lemma~\ref{lem:unipoch} gives the unicity of the function $Q_r$ and the compatibility of $Q_r$ and $Q_{r-1}$ under the restriction of the degree.

\begin{proof} We will work by induction on $r\geq 1$. The base of the induction is$r=1$ and is already proved in~\cite{Agg}. We assume now that $r\geq 2$ and that there exists a  (uniquely determined)  rational function in $Q_{r-1}\in \QQ[\pi^{2}]\otimes H^{\leq r-1}$ such that $v(\mu)=Q_{r-1}(g,n)+ O_{r-1}(g^{-r+1})$. 
We will first construct an element $Q_r\in \QQ[\pi^{2}]\otimes H^{\leq r}$ and then check that $v(\mu)=Q_r(g,n)+ O_{r}(g^{-r})$.

 Let us assume that $n\geq 2$. For all $i\geq 1$, we have $v(\mu^{(i)})=Q(g-i,n-1)+O_r(g^{-r-1})$.  Then we use the equality~\eqref{for:estwkappa} to write 
\begin{eqnarray}\label{for:sum}
v(\mu)-v(\mu^{(0)})= \sum_{i\geq 1} \left( \wk_i \frac{Q(g-i, n-1)}{(2g-3+n)_{2i}}\right) +O_r(g^{-r-1})
\end{eqnarray}
It is easy to check that for all $Q\in H$, the rational function $\frac{Q(g-i, n-1)}{(2g-3+n)_{2i}}$ is in $\bigoplus_{m,\ell\geq 2i} \QQ[\pi^{2}]\otimes H^{m,\ell}$. Therefore the right-hand side of the sum~\eqref{for:sum} is a rational function $\widetilde{D}_r$ in $\bigoplus_{m,\ell\geq 2} \QQ[\pi^{2}]\otimes H^{m,\ell}$. Moreover we define $D_r$ as the component of $\widetilde{D}_r$ in $H^{\leq r+1}$.  Now we define $\widetilde{Q}_r$ as $\Delta^{-1}(D_r)$.  

Proposition~\ref{prop:kappa} implies that there exists a polynomial $F_r\in \QQ_r[\pi^{2}][\frac{1}{g}]$ such that 
$$\widetilde{Q}_r(g,1)- v(2g-1)=F_r(g)+O(g^{-r}).
$$
We define $Q_r(g,n)=\widetilde{Q}_r(g,n)+F_r(g)\in H^{\leq r}$.

To check that $v(\mu)=Q_r(g,n)+O_r(g^{-r})$ we begin by checking that $v(2g-1)=Q_r(g,1)+O(g^{-r-1})$, which is obvious by construction of $Q_r$. Then for all $\mu\in {\rm Adm}(r)$ and all $1\leq i\leq n$, we define $\mu_i$ by induction as $\mu_1=\mu$ and $\mu_{i+1}=\mu_i^{(0)}$. Then we have:
$$
v(\mu)= v(2g-1)+ \sum_{i=2}^{n} v(\mu_{i-1})-v(\mu_{i}).
$$
Each summand in the right-hand side is equal to $D_r(g,n-i)$ up to a $O_r(g^{-r-1})$ and there are at most $2g-2$ summands. Thus
$$
v(\mu) = Q_r(2g-1,1) + \left( \sum_{i=2}^{n} D_r(g,n-i)\right) +O_r(g^{-r})=Q_r(2g-1,n)+O_r(g^{-r}).
$$
\end{proof}

\end{document}